\documentclass[reqno,11pt]{amsart}
\usepackage{a4wide,xcolor,eucal,enumerate,mathrsfs}
\usepackage[normalem]{ulem}
\usepackage{amsmath,amssymb,epsfig,amsthm} %,bbm
\usepackage[latin1]{inputenc}
\usepackage{psfrag}          % replace text in figures
\usepackage{tikz, animate}

\numberwithin{equation}{section}

\newtheorem{theorem}{Theorem}[section]

\newtheorem{corollary}[theorem]{Corollary}
\newtheorem{lemma}[theorem]{Lemma}
\newtheorem{proposition}[theorem]{Proposition}

\newtheorem{definition}[theorem]{Definition}

\newcommand{\N}{\mathbb{N}}

\newcommand{\be}{\begin{equation}}
\newcommand{\ee}{\end{equation}}

\newcommand{\R}{\mathbb{R}}

\renewcommand{\d}{{\mathrm d}}

\newcommand{\prob}[1]{\mathscr P(#1)}

\newcommand{\spt}{{\rm{spt}}}

\newcommand{\restr}[1]{\lower3pt\hbox{$|_{#1}$}}

\newcommand{\x}{\mathop{\textbf x}\nolimits}

\begin{document}

\title[Non-existence of optimal maps for the repulsive harmonic cost]
{Non-existence of optimal transport maps for the multi-marginal repulsive harmonic cost}
%Find a better title

\author{Augusto Gerolin}
\author{Anna Kausamo}\author{Tapio Rajala}
\address{University of Jyvaskyla\\
         Department of Mathematics and Statistics \\
         P.O. Box 35 (MaD) \\
         FI-40014 University of Jyvaskyla \\
         Finland}
\email{augustogerolin@gmail.com}		 
\email{anna.m.kausamo@jyu.fi}
\email{tapio.m.rajala@jyu.fi}

% \author{}
% \address{}
% \email{}

\thanks{The authors acknowledge the Academy of Finland projects no. 274372, 284511 and 312488.}
\subjclass[2000]{}%Primary 53C23. Secondary 28A33, 49Q20}
\date{\today}

\begin{abstract}
We give an example of an absolutely continuous measure $\mu$ on $\R^d$, for any $d \ge 1$,
such that no minimizer of the $3$-marginal harmonic repulsive cost with all marginals equal to $\mu$ is supported on a graph over the first variable.
\end{abstract}

\maketitle

\section{Introduction}
Let $\mu_1,\dots,\mu_N$ be probability measures in $\R^d$, $c \colon(\R^d)^N\to\R$ be a cost function and denote by $\Gamma(\mu_1,\dots,\mu_N)$ the set of probability measures $\gamma\in \prob{\R^{dN}}$ having marginals $\mu_1,\mu_2,\dots,\mu_N$. We are interested in understanding the structure of the minimizers of the multi-marginal Monge-Kantorovich problem 
\be\label{MKintro}
\inf_{\gamma\in\Gamma(\mu_1,\dots,\mu_N)}\int_{\R^{dN}}c(x_1,\dots,x_N) \,\d\gamma(x_1,\dots,x_N).
\ee

The problem has been studied recently in the literature for different cost functions $c$  motivated by problems in economics \cite{BeiHenPen, BeiCoxHue, BeiJul}, physics \cite{CFK17, CoDePDMa, GorSeiVig, LanDMaGerLeeGor, DMaGerNenGorSei} and  mathematics \cite{AguCar, GaSw, MoaPass,Pass14}. We refer to \cite{PassSurvey} for an extensive discussion on multi-marginal optimal transport and to \cite{DMaGerNen} for an introduction to the topic and for references in the context of optimal transport and density functional theory. 

Among the cost functions that are considered in the literature, we concentrate in this paper to the attractive harmonic or Gangbo-Swiech cost $c_a$ \cite{AguCar,GaSw} and to the repulsive harmonic cost $c_w$ \cite{DMaGerNen, SeiGorSav},
\[
c_a(x_1,\dots,x_N) = \sum^N_{i=1}\sum^N_{j=i+1}\vert x_j-x_i\vert^2, \quad c_w(x_1,\dots,x_N) = -\sum^N_{i=1}\sum^N_{j=i+1}\vert x_j-x_i\vert^2.
\]

Although the case $N=2$ is well-understood for a wide class of cost functions depending on the distance, thanks to the celebrated Brenier's theorem \cite{Bre91}, only partial results are currently known for the general case $N\ge3$. In fact, when $N=2$, if $h:\R^d\to\R^{+}$ is a strictly convex function, $c(x_1,x_2) = h(\vert x_1-x_2\vert)$, and $\mu_1$ is absolutely continuous, then the problem \eqref{MKintro} admits a unique solution $\gamma = (\text{id},T)_{\sharp}\mu_1$, which is concentrated on the graph of a map $T:\R^d\to\R^d$. 

The main theorem of this paper is stated below. For simplicity, we denote by $\Gamma_N(\mu)$ the set of probability measures in $\R^{dN}$ with all the $N$ marginals equal to $\mu$.

\begin{theorem}\label{thm:main}
Given $d \in \N$ there exists an absolutely continuous measure $\mu$ in $\R^d$ with bounded support such that
 there is no minimizer for the quantity
 \[
  \min_{\gamma \in \Gamma_3(\mu)} \int_{\R^{3d}} -|x_1-x_2|^2-|x_1-x_3|^2-|x_2-x_3|^2\,\d \gamma(x_1,x_2,x_3)
 \]
 that is induced by a map from one of the marginals. Moreover, there is a unique symmetric minimizer. 
\end{theorem}

We remark that when $\mu$ has finite second moments (or bounded support) and $N=2$ the problem \eqref{MKintro} corresponding to the cost function $c_w(x_1,x_2) = -\vert x_1-x_2\vert^2$ admits a unique solution $\gamma = (\text{id},G)_{\sharp}\mu$, where $G$ is the gradient of a concave function. This is a simple consequence of Brenier's theorem \cite{Bre91}. 

In \cite{MoaPass}, A. Moameni and B. Pass gave two general conditions on the cost functions in \eqref{MKintro} which ensure that any solution must concentrate on either finitely many or countably many graphs. Moreover, the authors exhibit two examples of cost functions and marginals $\mu_1,\mu_2,\mu_3$ such that there exists a unique minimizer which is concentrated in two graphs \cite[Example 4.2]{MoaPass} and \cite[Example 4.4]{MoaPass}. \medskip

\noindent
\textbf{A few words about the proof of Theorem \ref{thm:main}:} We will first construct a measure ${\mu}$ on $\R$ and after that show the general case $d\ge 1$. The measure ${\mu}$ will be defined as a sum ${\mu} = \mu^1 + \mu^2$ , where the measures $\mu^1$ and $\mu^2$ overlap so that this
forces the optimal couplings to be non-graphical. The support of each measure $\mu^k$ consists of three connected components that are separated so that any optimal coupling is forced to have one marginal in each of the components.  We use mainly two ingredients in the proof: (i) the fact that the support of an optimal transport $\gamma$ for the repulsive harmonic cost must be contained in the set $\lbrace x_1+\dots+x_N = k\rbrace$, where $k\in \R^d$; (ii) The classical $2$-marginal Brenier's theorem \cite{Bre91}. The main ideas are given in the proofs of Theorem \ref{thm:mainparts} and Lemma \ref{lma:discrete}. \medskip

\noindent
\textbf{Structure of the paper:} In Section \ref{sec:MOTAR}, we recall the main results regarding the attractive and repulsive harmonic cost. In particular, we show that in the $2$-marginal case the Monge-Kantorovich problem \eqref{MKintro} for both costs coincide under the hypothesis that the measure $\mu$ has finite second moments. Finally, in the Section \ref{sec:example}, we present the main construction of the paper, namely prove the counterexample for the existence of Monge-type solutions in the repulsive harmonic case.

\section{Multi-marginal optimal transport \\ for the attractive and repulsive harmonic costs}
\label{sec:MOTAR}

\subsection{The attractive harmonic cost}

In \cite{GaSw}, W. Gangbo and A. Swiech introduced the multi-marginal  optimal transport problem for the attractive harmonic cost (also known as the Gangbo-Swiech cost)
\be\label{costca}
c_a(x_1,\dots,x_N) =\sum^N_{i=1}\sum^N_{j=i+1}\vert x_i-x_j\vert^2,
\ee
and showed that the Monge-Kantorovich problem 
\be\label{MKAttractive}
\min_{\gamma\in\Gamma(\mu_1,\dots,\mu_N)} \int_{\R^{dN}}c_a(x_1,\dots,x_N)\,\d\gamma(x_1,\dots,x_N), 
\ee
admits a unique minimizer $\gamma_{opt}$ provided that the marginals $\mu_1,\dots,\mu_N$ have finite second moments and  are  not concentrated in \textit{small sets}. Moreover, $\gamma_{opt}$ is of Monge-type, that is, $\gamma_{opt} = (\text{id},T_1\dots,T_{N-1})_{\sharp}\mu_1$, where $T_i:\R^d\to\R^d$ satisfy $T_{i\sharp}{\mu_1} = \mu_i$, for all $i \in \{1,\dots,N-1\}$.

For the sake of completeness, we briefly present some heuristics on this result. First, by opening the squares of $c_a(x_1,\dots,x_N)$, we notice that studying the minimizers of \eqref{MKAttractive} is equivalent to studying the problem
\be\label{MKAttractive2}
\min_{\gamma\in\Gamma(\mu_1,\dots,\mu_N)} 2\int_{\R^{dN}}\sum^N_{i=1}\sum^N_{j=i+1}x_i\cdot x_j \,\d\gamma(x_1,\dots,x_N).
\ee

The main idea in \cite{GaSw} is to write the Kantorovich formulation of \eqref{MKAttractive2} as
\be\label{KAttractive}
\sup \bigg\lbrace \sum^N_{i=1}\int_{\R^d} u_i(x_i)\,\d\mu_i : \sum^N_{i=1} u_i(x_i) \leq \sum^N_{i=1}\sum^N_{j=i+1}x_i\cdot x_j \bigg\rbrace,
\ee
and to show that the above supremum is attained by some functions $u_i \in L^1(\mu_i)$. Then, letting $\gamma$ be an optimal transport plan for \eqref{MKAttractive2}, and using the optimality conditions in \eqref{KAttractive}, we see that the Kantorovich potentials $u_i$ must satisfy the equality
\be\label{eq:Koptcond}
u_1(x_1)+\cdots+u_N(x_N) =\sum^N_{i=1}\sum^N_{j=i+1}x_i\cdot x_j, \text{ on } \spt(\gamma).
\ee
Now, applying the operators $\nabla_j$, $j \in\{1,\dots,N\}$, on both sides of \eqref{eq:Koptcond}, one aims at solving the following system of equations
\[
\begin{cases}
\nabla u_1(x_1) = \sum^N_{i=2}x_i \\
\cdots \\
\nabla u_j(x_j) = \sum^N_{i=1,i\neq j}x_i \\
\cdots \\
\nabla u_N(x_N) = \sum^{N-1}_{i=1}x_i \\
\end{cases} \quad \text{ or, equivalently, } \quad \begin{cases}
 x_1 + \nabla u_1(x_1) = \sum^N_{i=1}x_i \\
\cdots \\
 x_j + \nabla u_j(x_j) = \sum^N_{i=1}x_i \\
\cdots \\
 x_N + \nabla u_N(x_N) = \sum^N_{i=1}x_i \\
\end{cases}, \medskip
\]
which is handily solvable provided that one can invert $\nabla_{x_j} (\vert \cdot \vert^2 + u_j)(x_j)$ for all $j\in\{2,\dots,N\}$.
Indeed, we have that $\nabla_{x_1} (\vert \cdot \vert^2 + u_1)(x_1) = \nabla_{x_j} (\vert \cdot \vert^2 + u_j)(x_j)$ for all $j \in\{2,\dots,N\}$. Therefore, $x_j = \nabla_{x_j}^{-1} (\nabla_{x_1}(\vert \cdot \vert^2 + u_1)(x_1)) =: T_j(x_1)$. In fact, one can show that the functions $u_i(x_i) + \frac{1}{2}\vert x_i\vert^2$ are strictly convex and thus $\nabla_{x_j}(\frac{1}{2}\vert \cdot\vert^2 + u_j)(x_j)$ admits an inverse which is its convex conjugate $\nabla^*_{x_j}$ (see \cite{GaSw}, for details). 

The above approach leads to the following result by Gangbo-Swiech \cite{GaSw}.
\begin{theorem}[Gangbo-Swiech, \cite{GaSw}]\label{thm_GS}
Let $\mu_1,\dots,\mu_N$ be non-negative Borel probability measures in $\R^d$ vanishing on $(d-1)-$rectifiable sets and having finite second moments, and let $c_a$ be the attractive harmonic cost \eqref{costca}. Then 
\begin{itemize}
\item[(1)] the problem \eqref{KAttractive} admits maximizers $u_i$ which are $\mu_i$-differentiable.
\item[(2)] There exists a unique minimizer $\gamma_{opt}$ in \eqref{MKAttractive} such that $\gamma_{opt} = (\text{id},T_1,\dots,T_{N-1})_{\sharp\mu_1}$, where $T_i\colon\R^d\to\R^d$ are defined by $T_i(x_1) = \nabla^*_{x_i}(\nabla_{x_1}(\vert \cdot\vert^2 + u_1)(x_1)$ for $i=2,\dots,N$. 
\end{itemize}
\end{theorem}

Theorem \ref{thm_GS} was extended by H. Heinrich \cite{Hei2002}, allowing the inclusion of more general cost functions $c(x_1,\dots,x_N) = l(x_1+\cdots+x_N)$ with $l\colon\R^d\to\R^{+}$ concave. M. Agueh and G. Carlier \cite{AguCar} remarked that the attractive harmonic cost is equivalent to the so-called \textit{barycenter problem} in multi-marginal optimal transport, which has many applications in inverse problems in imaging sciences.  We refer to the survey \cite{PassSurvey} and the references therein for further applications and extensions of the attractive harmonic cost.

\subsection{The repulsive harmonic cost}
In this section, we consider cost functions of the type
\[
c_w(x_1,\dots,x_N) = \sum^N_{i=1}\sum^N_{j=i+1} -\vert x_i-x_j\vert^2,
\]
or, more generally, 
\[
c_h(x_1,\dots,x_N) = h(x_1+x_2+\cdots+x_N), \text{ where } h:\R^d\to\R^{+} \text{ is a convex function.}
\]  Costs in this class are said to be \textit{repulsive} because optimal transport plans must place points $x_i$ as far as possible from each other. We show in Subsection \ref{subsec:N=2}, that in the case where the number of marginals $N$ is equal to $2$, the attractive and repulsive costs are equivalent via Brenier's theorem. In particular, the Monge-Kantorovich problem associated with the repulsive harmonic cost admits a unique minimizer which is of Monge-type. 

This situation, however, changes drastically in the multi-marginal case where generally there is no hope of uniqueness of optimal transport plans for the Monge-Kantorovich problem associated with the cost $c_w$.  Moreover, examples of diffuse-like, fractal-like, and Monge-like solutions can be constructed \cite{DMaGerNen}. 

Although the most interesting cost in the class of repulsive costs is the Coulomb one, since it plays an important role in density functional theory \cite{GorSeiVig, SeiGorSav},  
the repulsive harmonic cost has been used as  a toy-model to study problems in that context.  As observed by physicists \cite{SeiGorSav, CorKarLanLee}, it  has several advantages compared to the Coulomb one, since easy Monge-type solutions can be constructed in any dimension $d$ for specific densities. Additionally, the repulsive harmonic cost admits a very simple characterization for optimality (see Proposition \ref{prop:optimalRA}). 
 
We now recall a few results and central examples in the repulsive harmonic cost case. First, we notice that the Monge-Kantorovich problem associated with $c_w$ is equivalent to the one associated with $c_h(x_1,\dots,x_N) = h(x_1+\dots+x_N)$ when $h(z) = \vert z\vert^2$.

\begin{proposition}
Assume that $\mu_1,\dots,\mu_N$ are probability measures in $\R^d$ with finite second moments. Then,
$$ \underset{\gamma\in\Gamma(\mu_1,\dots,\mu_N)}{\operatorname{argmin}}\int_{\R^{dN}} c_w(x_1,\dots,x_N)\,\d\gamma = \underset{\gamma\in\Gamma(\mu_1,\dots,\mu_N)}{\operatorname{argmin}}\int_{\R^{dN}} \vert x_1+\cdots+x_N\vert^2 \,\d\gamma. $$
\end{proposition}

As noticed in \cite{DMaGerNen}, the Monge-Kantorovich problem for cost functions $c_h$ have a simple characterization.

\begin{proposition}[\cite{DMaGerNen}]
\label{prop:optimalRA}
Assume that $\mu_1,\dots,\mu_N$ are probability measures in $\R^d$ with finite second moments and $c_h(x_1,\dots,x_N) = h(x_1+\cdots + x_N)$, where $h\colon\R^d\to \R^+$ is a convex function. Then, if there exists a plan $\gamma$ concentrated in some hyperplane $\lbrace x_1+\cdots+x_N = k\rbrace$, $k\in\R^d$, then $\gamma$ is an optimal transport plan for the problem
\be\label{MKRepulsiveHarmonic}
\min_{\gamma\in\Gamma(\mu_1,\dots,\mu_N)}\int_{\R^{dN}} h(x_1+\cdots+x_N) \,\d\gamma(x_1,\dots,x_N).
\ee
In this case, $\tilde{\gamma} \in \Gamma(\mu_1,\dots,\mu_N)$ is optimal in \eqref{MKRepulsiveHarmonic}, if and only if,  
\begin{equation}\label{eq:optimality}
\spt(\tilde{\gamma}) \subset \lbrace (x_1,\dots,x_N) \in \R^{dN} \,|\, x_1+\cdots +x_N = k \rbrace. 
\end{equation}
In particular, the constant vector $k$ can be computed explicitly as
\[
k = h\left(\sum^N_{j=1}\int x \,\d\mu_j(x)\right).
\]
\end{proposition}

In the multi-marginal setting, the optimality condition 
\[\spt(\gamma) \subset \left\{(x_1,\ldots,x_N) \in \R^{dN} ~|~ x_1+\dots+x_N = k\right\}\] is much easier to handle than the $c$-cyclical monotonicity and the optimality conditions given by the Kantorovich problem. 

We list some examples of minimizers in \eqref{MKRepulsiveHarmonic} that can be constructed via Proposition \ref{prop:optimalRA}. An interesting case for repulsive costs is when all the $N$ marginals are the same and given by an absolutely continuous measure $\mu = \rho\mathcal{L}^d$. \medskip

\begin{itemize}

\item[-] \textit{Diffuse optimal plan or ``fat'' plan:} 

\noindent
Suppose $d=1$, $N=3$ and the measure $\mu$ is defined by $\mu=\frac 12 \mathcal{L}|_{[-1,1]}$. Let $\gamma =\frac 12 \mathcal{H}^{2}|_{ H}  g(\max \{ |x_1| , | x_2| , |x_3| \}) $, where $H$ is defined as $H= \{ x_1+x_2+x_3=0 \} \cap \{ |x| \leq 1, |y| \leq 1, |z| \leq 1 \}$, $g(x) = \frac{\sqrt{3}}{6}x$ and $\mathcal{H}^2$ denotes the $2$-dimensional Hausdorff measure. Then, by Proposition \ref{prop:optimalRA}, $\gamma$ is an optimal plan, since $\gamma \in \Pi_3(\mu)$ and $k= 3\int x d\mu = 0$. Moreover, this plan is not concentrated on a graph of a map \cite{DMaGerNen}. \bigskip

\item[-] \textit{Fractal-like optimal plan concentrated on a graph of a map \cite{DMaGerNen}:} 

\noindent
Let $\mu = \mathcal{L}^d|_{[0,1]^d}$ be the uniform measure in the $d$-dimensional unit cube and $N\geq 3$. Every point $z\in[0,1]$ can be represented in its $N$-th base. Namely, $z=\sum^{\infty}_{k=1}a_k/N^k$ with $a_k \in \lbrace 0,\dots,N-1\rbrace$. We define a map $T$ by permuting  the $N$ symbols of $a_k$.
\begin{align*}
T:z\in[0,1]\mapsto T(z) = \sum^{\infty}_{k=1} \mathcal{S}(a_k)/N^K \in [0,1], \\
\text{ where } \mathcal{S}(i) = i+1 \text{ and } \mathcal{S}(N-1) = 0.  
\end{align*}
\noindent
One can show that \cite{DMaGerNen}, if $x_1 = (z_1,\dots,z_N)$ the map $\tilde{T}(x_1) = (T(z_1),T(z_2),\dots,T(z_N))$ preserves the uniform measure in the $d$-dimensional cube and has the property that 
\[
x_1 + \tilde T(x_1) + \cdots + \tilde T^{(N-1)}(x_N) = N/2 = N\left(\int_{\R^d}x_1\,\d\mu(x_1)\right)^2.
\]
Therefore, by Proposition \ref{prop:optimalRA} the plan $\gamma = (\text{id}, \tilde T, \dots,\tilde T^{(N)})_{\sharp}\mu$ is optimal. The map $\tilde{T}$ is not continuous at any point \cite{DMaGerNen}. \bigskip

\item[-] \textit{Optimal plan concentrated on the graph of a regular cyclical map $(N \text{ even})$:}
\noindent
Suppose that, for all $i=1,\dots,2k$, $\mu_i = \mu=\rho\mathcal{L}^d$ is such that $\rho(\vert x\vert) = \rho(x)$ for all $x\in\R^d$. Then $\gamma(x_1,\dots,x_{2k}) = (\text{id},T,T^{(2)},\dots, T^{(2k-1)})_{\sharp}\mu$ where $T\colon x\in\R^d \mapsto -x \in \R^d$ is an optimal plan for \eqref{MKRepulsiveHarmonic}. In fact, 
$$N\int_{\R^d}x \rho(x)\,\d x = 0$$
and 
$$x + T(x) + T^{(2)}(x) + \cdots + T^{(N-1)}(x) = 0.$$
\end{itemize}

\subsection{Discussion on the $N=2$ case and Brenier's Theorem}\label{subsec:N=2}

We start by recalling the classical Brenier's theorem in optimal transport. Notice that the attractive harmonic cost is a natural generalization of the distance squared cost when $N>2$.

\begin{theorem}[Brenier, \cite{Bre91}]
Let $\mu_1$ and $\mu_2$ be Borel probability measures in $\R^d$ with finite second moments. Assume that $\mu_1$ is absolutely continuous with respect to the Lebesgue measure. Then the problem
\[
\min_{\gamma\in\Gamma(\mu_1,\mu_2)}\int_{\R^d\times\R^d}\vert x_1-x_2\vert^2 \,\d\gamma(x_1,x_2),
\]
admits a unique minimizer $\gamma_{opt} = (\text{id},T)_{\sharp} \mu_1$, where $T(x_1) = \nabla \phi(x_1)$, and $\phi\colon\R^d\to\R$ is a convex function. As a consequence, the Monge problem with squared distance cost 
\[
\min \left\lbrace \int_{\R^d}\vert x_1-T(x_1)\vert^2\,\d\mu_1(x_1) \,:\, T\colon \R^d\to\R^d, T_{\sharp}\mu_1=\mu_2\right\rbrace, 
\]
admits a unique minimizer.
\end{theorem}

We call the optimal map $T$ the \textit{Brenier's map}. In the two marginal case,
a solution of the Monge-Kantorovich problem for the repulsive harmonic cost is an immediate consequence of Brenier's theorem. 

\begin{corollary}\label{cor:RAN2}
Let $\mu_1$ and $\mu_2$ be Borel probability measures in $\R^d$ having finite second moments. Assume that $\mu_1$ is absolutely continuous with respect to the Lebesgue measure. Then the problem
\[
\min_{\gamma\in\Gamma (\mu_1,\mu_2)}\int_{\R^d\times\R^d}-\vert x_1-x_2\vert^2 \,\d\gamma(x_1,x_2),
\]
admits a unique minimizer $\gamma_{opt} = (\text{id},T)_{\sharp}\mu_1$, where $T(x_1) = \nabla \psi(x_1)$, and $\psi:\R^d\to\R$ is a concave function.
\end{corollary}

\begin{proof}
First notice that, since $\mu_1$ and $\mu_2$ have finite second moments,
\begin{align*}
2\int_{\R^d}\vert x_1\vert^2\,\d\mu_1 + & 2\int_{\R^d}\vert x_2\vert^2\,\d\mu_2 + \min_{T_{\sharp}\mu_1=\mu_2} \int_{\R^d} -\vert x_1-T(x_1)\vert^2\,\d\mu_1 \\
& = \min_{G_{\sharp}\mu_1 =\tilde{\mu_2}} \int_{\R^d} \vert x_1-G(x_1)\vert^2\,\d\mu_1,
\end{align*}
where $G=-T$ and $\tilde{\mu_2} = (-\text{id})_{\sharp}\mu_2$. By Brenier's theorem the problem on the right-hand side has a unique solution $G = \nabla \phi$, with $\phi\colon\R^d\to\R$ a convex function. Then, $T = \nabla \psi$ is an optimal map for the repulsive harmonic cost, which is the gradient of a concave function $\psi = -\phi$.
\end{proof}

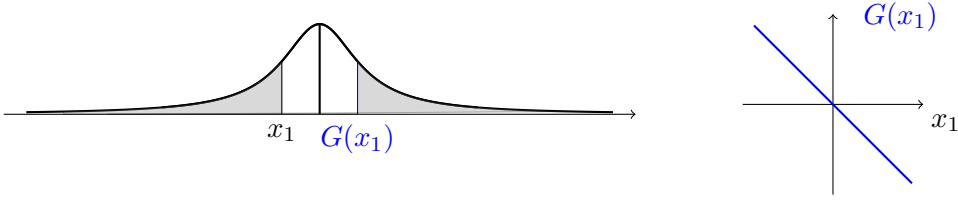
\begin{figure}[http]\label{fig.lorenzian2d}
\begin{minipage}{.60\textwidth}
\begin{center}
\begin{tikzpicture}[scale=0.6]
\draw [smooth, samples=100, domain=-6.5:6.5,thick] plot(\x,{2/((\x)*(\x)+1)});
\draw[->] (-7,0) -- (7,0);

  \def\offset{15};

\foreach \x in {1} {\draw[ thick] ( { tan( -90 +180*(\x)/2 )  },0) -- ( { tan( -90 +180*(\x)/2 )  },{2/(tan( -90 +180*(\x)/2 )*tan( -90 +180*(\x)/2 ) +1 )});}
\draw [smooth, samples=100, domain=-3.5:3.5,thick] plot(\x,{2/((\x)*(\x)+1)});

\draw[black] ( { tan( -95+\offset+80/2 )  },0) node[below] {$x_1$} -- ( { tan( -95 +\offset +80/2 )  },{2/(tan( -95 +\offset +80/2)*tan( -95 +\offset +80/2) +1)});

\draw[blue] ( { tan( 95-\offset-80/2 )  },0) node[below] {$G(x_1)$} -- ( { tan( 95 -\offset -80/2 )  },{2/(tan( 95 -\offset -80/2)*tan( 95 -\offset -80/2) +1)});

\draw[fill=black!50, opacity=0.3] plot[smooth, samples=100, domain= {tan( -95+\offset )}:{ tan(-95+\offset+80/2)}] ( \x,{2/((\x)*(\x)+1)}) |- ({tan(-95+\offset)},0) -- cycle;

\draw[fill=black!50, opacity=0.3] plot[smooth, samples=100, domain= {tan( -95+\offset )}:{ tan(-95+\offset+80/2)}] ( -\x,{2/((\x)*(\x)+1)}) |- ({tan(-93+\offset)},0) -- cycle;

\end{tikzpicture}
\end{center}
\end{minipage}
\begin{minipage}{.3\textwidth}
\begin{center}
\begin{tikzpicture}[scale=0.3]
\draw[->] (-4,0) -- (4,0);
\draw[->] (0,-4) -- (0,4);

  \def\offset{15};

\draw ({5},0) node[below] {$x_{1}$}; 
\draw[blue] ({3},5) node[below] {$G(x_{1})$};

\draw [blue, smooth, samples=100, domain=-3.50:-0.01,thick] plot(\x,{(-\x)});
\draw [blue, smooth, samples=100, domain=0.01:3.5,thick] plot(\x,{(-\x)});
\end{tikzpicture}
\end{center}
\end{minipage}
\caption{Example of an optimal map given by Corollary \ref{cor:RAN2} in the one-dimensional case $d=1$: $\mu = \mu_1 = \mu_2$ is the Gaussian-like shape density in the left-side picture. The map $G\colon\R\to\R$  is the anti-monotone map $G(x) = -x$ (right-hand picture).}
\end{figure}

\section{Repulsive Harmonic cost: Monge is equal to Monge-Kantorovich?}\label{sec:example}

Let $\mu$ be an absolutely continuous probability measure with a finite second moment. It is natural to inquire if there exists a Monge-type minimizer for the Monge-Kantorovich problem $(N\geq 3)$
\[
\min_{\gamma\in\Gamma_N(\mu)}\int_{\R^{dN}}c_w(x_1,\dots,x_N)\,\d\gamma(x_1,\dots,x_N),
\] 
or, for the cost $c_h(x_1,\dots,x_N) = h(x_1+\cdots+x_N)$, with $h\colon\R^d\to\R^+$ strictly convex.

A particular case of a theorem of A. Pratelli \cite{Pra07} states that the minimum of the Monge-Kantorovich problem coincides with the infimum of the Monge problem for a large class of cost functions, including Coulomb and $c_h$ cost functions. 

\begin{theorem}[A. Pratelli, \cite{Pra07}] 
Assume that $\mu$ is a Borel probability measure without atoms and $c\colon(\R^d)^N\to\R\cup\lbrace +\infty\rbrace$ is a continuous and bounded from below cost function. Then,
\be\label{prop:pratelli}
\min_{\gamma\in\Gamma_N(\mu)}\int_{\R^{dN}}c\,\d\gamma = \inf \Big\lbrace \int_{\R^{d}}c(x,T_1(x),\dots,T_{N-1}(x))\,\d\mu(x) : \begin{array}{
l} T_{i{\sharp}}\mu=\mu \end{array} \Big\rbrace,
\ee
where $T_i \colon \R^d\to\R^d$ are Borel functions and $i\in \{1,\dots,N-1\}$.
\end{theorem}

In other words, we are interested in knowing if the above infimum in \eqref{prop:pratelli} is achieved. The following counterexample proves that this is generally not the case for the repulsive harmonic cost $c_w$ or, more generally, for cost functions  $c_h(x_1,\dots, x_N) = h(\sum x_i)$ depending on a convex function $h\colon\R^d\to\R$.

Let us recall the notion of a symmetric transport plan.

\begin{definition}[Symmetric measures]\label{def:symmeasures}
A measure $\gamma \in \mathcal{P}(\R^{dN})$ is symmetric if 
\[
\int_{\R^{dN}}g(x_1,\dots,x_N)\,\d\gamma = \int_{\R^{dN}} g(\sigma(x_1,\dots,x_N))\,\d\gamma, \text{ for all } g \in  \mathcal{C}(\R^{dN})
\]
and for all permutations of $N$ symbols $\sigma \in \mathfrak{S}_N$. We denote by $\Gamma_N^{sym}(\mu)$, the space of all $\gamma \in \Gamma_N(\mu)$ which are symmetric.
\end{definition}

The following result follows by observing that
\begin{itemize} 
\item[(i)]  $ \Gamma_N^{sym}(\mu)\subset \Gamma_N(\mu)$; 
\item[(ii)] if $\gamma \in \Gamma_N(\mu)$ then $\gamma_{sym} = \frac{1}{N!}\sum_{\sigma \in \mathfrak{S}_N}\sigma_{\sharp}\gamma \in  \Gamma_N^{sym}(\mu)$.
\end{itemize}

\begin{proposition}\label{prop:symplans}
If $\mu$ is an absolutely continuous measure in $\R^d$ with finite second moments, then
\begin{equation}\label{eq:MKsymMK}
\min_{ \gamma \in \Gamma_N(\mu) } \int_{\R^{dN}} c_w(x_1,\ldots, x_N) \, \d  \gamma = \min_{ \gamma \in \Gamma_N^{sym}(\mu) } \int_{\R^{dN}} c_w(x_1,\ldots, x_N) \, \d  \gamma. \bigskip
\end{equation}
\end{proposition}

\subsection{Construction of the counterexample}
The main goal of this section is to exhibit an absolutely continuous measure $\mu$ such that, for the repulsive harmonic problem \eqref{MKRepulsiveHarmonic}, there is no Monge-type minimizer. We recall our main theorem.

\begin{theorem}\label{thm:mainre}
 Given $d \in \N$ there exists an absolutely continuous measure $\mu$ in $\R^d$ with bounded support such that
 there is no minimizer for the quantity
 \begin{equation}\label{eq:mainclaimre}
  \min_{\gamma \in \Gamma_3(\mu)} \int_{\R^{3d}} -|x_1-x_2|^2-|x_1-x_3|^2-|x_2-x_3|^2\,\d \gamma(x_1,x_2,x_3)
 \end{equation}
 that is induced by a map from one of the marginals. Moreover, there is a unique symmetric minimizer.
\end{theorem}

Let us first define the absolutely continuous  measure $\mu$ in $\R$, since all the ideas of the construction are
present already in the one-dimensional case. Then, in Subsection \ref{subsec:dge1} we
give the necessary modifications for the higher dimensions.

The measure $\mu$ can be understood as a sum $\mu = \mu^1 + \mu^2$, where the support of both measures $\mu^k$ consists of three connected components that are separated so that any optimal coupling is forced to have one marginal in each of the components.
Let us now define the connected components.
Let
\[
 C = \left[0,1/2\right],
\]
and for $k \in \{1,2\}$, let 
\[
R^k = \left[3^k, 3^k+1/2\right]
\qquad
\text{and}
\qquad
 L^k = \left[-3^k-1,-3^k \right].
\]
Using these we define $R = R^1 \cup R^2$ and $L = L^1\cup L^2$.

In order to study the structure of the minimizers, let us consider $\mu$ as the a sum 
\begin{equation}\label{mucounterexample}
 \mu = \frac{1}{3}\left(\mu_L + \mu_C + \mu_R\right),
\end{equation}
where different components correspond to the measure on the left, center and right, respectively.
We define these parts as
\begin{equation}\label{mucounterexampleparts}
 \mu_L = \frac12 \mathcal{L}\restr{L}, \qquad \mu_C = 2\mathcal{L}\restr{C}, \qquad \mu_R = \mathcal{L}\restr{R}.
\end{equation}

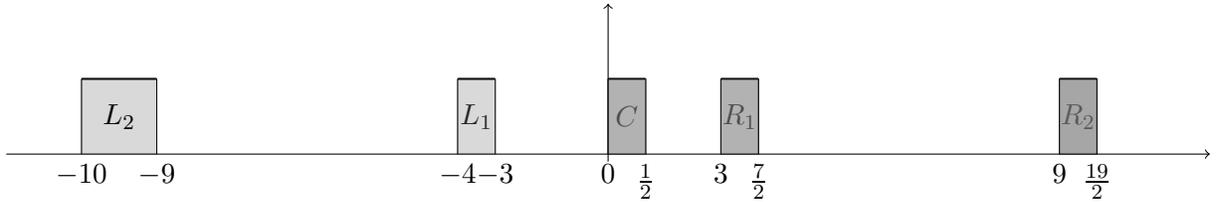
\begin{figure}[http]
\begin{center}
\begin{tikzpicture}[scale=1]
\draw [smooth, samples=100, domain=-7:-6,thick] plot(\x,{1});
\draw [smooth, samples=100, domain=-2:-1.5,thick] plot(\x,{1});

\draw [smooth, samples=100, domain=0:0.5,thick] plot(\x,{1});

\draw [smooth, samples=100, domain=1.5:2,thick] plot(\x,{1});

\draw [smooth, samples=100, domain=6:6.5,thick] plot(\x,{1});

\draw[->] (-8,0) -- (8,0);
\draw[->] (0,-0.1) -- (0,2);

  \def\offset{15};

\draw ( { -7 },0) node[below] {$-10$} -- ( { -7  },{+1)}); 
\draw ( { -6 },0) node[below] {$-9$} -- ( { -6  },{+1)}); 
\draw ( { -2 },0) node[below] {$-4$} -- ( { -2  },{+1)}); 
\draw ( { -1.5 },0) node[below] {$-3$} -- ( { -1.5  },{+1)}); 
\draw ( { 0 },0) node[below] {$0$} -- ( { 0  },{+1)}); 
\draw ( { 0.5 },0) node[below] {$\frac{1}{2}$} -- ( { 0.5  },{+1)}); 
\draw ( { 1.5 },0) node[below] {$3$} -- ( { 1.5  },{+1)}); 
\draw ( { 2 },0) node[below] {$\frac{7}{2}$} -- ( { 2  },{+1)}); 
\draw ( { 6 },0) node[below] {$9$} -- ( { 6  },{+1)}); 
\draw ( { 6.5 },0) node[below] {$\frac{19}{2}$} -- ( { 6.5  },{+1)}); 

\draw[fill=black!50, opacity=0.3] plot[smooth, samples=100, domain= {-7)}:{ -6)}] ( \x,{1}) |- (-7,0) -- cycle;
\node[] at (-6.5,0.5) {$L_2$};
\node[] at (-1.75,0.5) {$L_1$};

\node[] at (0.25,0.5) {$C$};

\node[] at (1.75,0.5) {$R_1$};

\node[] at (6.25,0.5) {$R_2$};

\draw[fill=black!50, opacity=0.3] plot[smooth, samples=100, domain= {-1.5)}:{ -2)}] ( \x,{1}) |- (-1.5,0) -- cycle;

\draw[fill=black!50, opacity=0.6] plot[smooth, samples=100, domain= {0.5)}:{ 0)}] ( \x,{1}) |- (0.5,0) -- cycle;

\draw[fill=black!50, opacity=0.6] plot[smooth, samples=100, domain= {1.5)}:{2)}] ( \x,{1}) |- (1.5,0) -- cycle;

\draw[fill=black!50, opacity=0.7] plot[smooth, samples=100, domain= {6)}:{6.5)}] ( \x,{1}) |- (6,0) -- cycle;

\end{tikzpicture}
\end{center}
\caption{Sketch of the support of $\mu$.}
\end{figure}

\subsection{Unique non-graphical minimizer in $\Gamma(\mu_C,\mu_R,\mu_L)$}

Let us first state a version of Theorem \ref{thm:mainre} where the marginals are given by $\mu_C,\mu_R,\mu_L$.
Theorem \ref{thm:mainparts} contains the main ideas in the proof of Theorem \ref{thm:mainre}; the existence of a unique non-graphical minimizer.

\begin{theorem}\label{thm:mainparts}
 Let $\mu_L,\mu_C,\mu_R$ be the absolutely continuous measures defined in \eqref{mucounterexampleparts}. Then there exists
 a unique  minimizer $\gamma_0$ of
 \begin{equation}\label{eq:mainclaimparts}
  \min_{\gamma \in \Gamma(\mu_C,\mu_R,\mu_L)} \int_{\R^{3d}} -|x_1-x_2|^2-|x_1-x_3|^2-|x_2-x_3|^2\,\d \gamma(x_1,x_2,x_3)
 \end{equation}
 which is given by
 \[
  \gamma_0 = \frac12\sum_{k=1}^2\left((H_k)_\sharp\mathcal{L}\restr{C}\right),
 \]
 where
 \[
 H_k(x) = (x, x+3^k, -2x - 3^k).
\]
In particular, this minimizer is not induced by a map from the first coordinate.
\end{theorem}

\begin{proof}
By construction, the projections of $\gamma_0$ are $\mu_C$, $\mu_R$ and $\mu_L$. Moreover, since for all $k \in \{1,2\}$ and $x \in \R$, writing $H_k(x) = (x_1,x_2,x_3)$, we have
\begin{equation}\label{eq:zerosumpart}
 x_1 + x_2 + x_3 = x + x + 3^k-2x-3^k = 0,
\end{equation}
we get from Proposition \ref{prop:optimalRA} that the measure $\gamma_0$ is a minimizer. 
Since the disintegration of $\gamma_0$ with respect to the projection to the first coordinate is for $\mu_C$-almost every $x \in \R$ unique and equal to
\[
 \frac12\delta_{(x, x+3, -2x - 3)} +  \frac12\delta_{(x, x+9, -2x - 9)},
\]
we conlculde that $\gamma_0$ is not induced by a map from the first coordinate. 

What remains to show is that $\gamma_0$ is the unique minimizer in \eqref{eq:mainclaimparts}. Let $\gamma$ be a minimizer of \eqref{eq:mainclaimparts}. Since for $\gamma_0$  we have by \eqref{eq:zerosumpart} that
\[
 \int_{\R^3}|x_1+x_2+x_3|^2\,\d\gamma_0(x_1,x_2,x_3) = 0,
\]
we conclude that the same must hold for $\gamma$. Consequently, for any $(x_1,x_2,x_3) \in \spt(\gamma)$ we have 
\begin{equation}\label{eq:zerosumagain}
x_1+x_2 + x_3 = 0. 
\end{equation}

The idea of the proof is now to look at the variational problem 
\[
\min \bigg\lbrace \int |x_1+x_3|^2\,\d\omega(x_1,x_3) \,\Big|\, \omega \in \Gamma(\mu_L,\mu_R) \bigg\rbrace.
\]

By Corollary \ref{cor:RAN2}, there exists a unique minimizer $\omega_a = (\text{id},G)_{\sharp}\mu_L$, which is given by the anti-monotone map $G$.
 Let $T(x_1,x_3) = -x_1-x_3$. 
Since $\mathcal{L}\restr{C} = T_\sharp(\omega_a)$  and
 \[
  \int |x_2|^2\,\d T_\sharp(\omega)(x_2) = \int |x_1+x_3|^2\,\d\omega(x_1,x_3),
 \]
 we have that $\omega_a$ is the unique coupling of $\mu_L$ and $\mu_R$ that is pushed to $\mathcal{L}\restr{C}$ by $T$.
 Recall that by Proposition \ref{prop:optimalRA}, we have $x_2 = -x_1-x_3$ for any $(x_1,x_2,x_3) \in \spt(\gamma)$.
Therefore, $\gamma = S_\sharp(\omega_a) = \gamma_0$, with $S(x_1,x_3) = (x_1,-x_1-x_3,x_3)$. We conclude that $\gamma_0$ is the unique minimizer of \eqref{eq:mainclaimparts}. 
\end{proof}

\subsection{Structure of the minimizers in in $\Gamma_3(\mu)$}

Let us now study the minimizers in the case where all the three marginals are $\mu$.
We start with the following lemma.

\begin{lemma}\label{lma:discrete}
 Let $\gamma$ be a minimizer in \eqref{eq:mainclaimre} and $(x_1,x_2,x_3) \in \spt(\gamma)$. Then there exist $k \in \{1,2\}$ and a permutation $\sigma$ of $\{1,2,3\}$ so that
 \begin{equation}\label{eq:onefromeachblock}
  x_{\sigma(1)} \in L^k, \qquad x_{\sigma(2)} \in C, \qquad x_{\sigma(3)} \in R^k.
 \end{equation}
\end{lemma}
\begin{proof}
 We prove the claim by ruling out the other possibilities. 
 By applying a permutation, if necessary, we may assume that $x_1 \le x_2 \le x_3$.
 
 If $x_1,x_2,x_3 \in C \cup R^1 \cup R^2$, then \eqref{eq:zerosumagain} forces $x_1 = x_ 2 = x_3 = 0$. 
 Then by the definition of the support and the fact that $\mu(\{0\}) = 0$, there
 exist points $(\tilde x_1,\tilde x_2,\tilde x_3) \in \spt(\gamma)\setminus \{(x_1,x_2,x_3)\}$ arbitrarily close to $(x_1,x_2,x_3)$. For close enough points, the form of $\mu$ then gives that $(\tilde x_1,\tilde x_2,\tilde x_3) \in C^3$ and thus $\tilde x_1 +\tilde x_2+ \tilde x_3 > 0$, contradicting \eqref{eq:zerosumagain}. Thus  $x_1 \in L^1 \cup L^2$.
 
 If now $x_2 \in L^1 \cup L^2$, then by \eqref{eq:zerosumagain},
 \[
  x_3 \in [6,8] \cup [12,14] \cup [18,20].
 \]
 But then, $x_3 \notin \spt(\mu)$, giving a contradiction. Therefore, $x_2,x_3 \in C \cup R^1 \cup R^2$.

 If $x_2,x_3 \in C$, then $x_1+x_2+x_3 < 0$. Thus, $x_3 \in R^1 \cup R^2$. In order to $x_2 \in \spt(\mu)$, the only possibility is then to have
 $x_1 \in L^1$ and $x_3 \in R^1$, or $x_1 \in L^2$ and $x_3 \in R^2$. In both cases, one must have $x_2 \in C$. This proves the claim.
\end{proof}

\begin{proof}[Proof of Theorem \ref{thm:mainre}]
Let $\gamma$ be a minimizer in \eqref{eq:mainclaimre}. Let us translate the problem to the setting of Theorem \ref{thm:mainparts} using the following mapping $F \colon \R^3 \to \R^3$ defined as
\[
F((x_1,x_2,x_3)) = 
\begin{cases}
 (x_3,x_2,x_1), & \text{if }(x_1,x_2,x_3) \in L \times R \times C,\\
 (x_3,x_1,x_2), & \text{if }(x_1,x_2,x_3) \in R \times L \times C,\\
 (x_2,x_3,x_1), & \text{if }(x_1,x_2,x_3) \in L \times C \times R,\\
 (x_2,x_1,x_3), & \text{if }(x_1,x_2,x_3) \in R \times C \times L,\\
 (x_1,x_3,x_2), & \text{if }(x_1,x_2,x_3) \in C \times L \times R,\\
 (x_1,x_2,x_3), & \text{otherwise}.
\end{cases}
\]
The map $F$ permutes the coordinates differently in different parts of the space so that $\gamma$-almost every point will be mapped into $C \times R \times L$, due to Lemma \ref{lma:discrete}.
Since $F$ is pointwise just a permutation, we have that $\gamma$-almost every $(x_1,x_2,x_2) \in \R^3$ the point $(y_1,y_2,y_3) = F((x_1,x_2,x_3))$ satisfies
$$y_1 + y_2 + y_3 = 0.$$ Therefore, by Proposition \ref{prop:optimalRA}, the measure $F_\sharp \gamma$ is an optimal coupling (for its marginal measures).
We claim that $F_\sharp \gamma \in \Pi(\mu_C,\mu_R,\mu_L)$. 
Let us check this for the first marginal. Let $A \subset \R$. Then by Lemma \ref{lma:discrete} and the definition of $F$, we get
\begin{align*}
F_\sharp\gamma (A\times \R\times \R) & = F_\sharp\gamma((A\cap C)\times \R\times \R) \\
& = \gamma((A \cap C)\times \R \times \R) + \gamma(\R \times (A \cap C)\times \R) + \gamma( \R \times \R \times (A \cap C))\\
& = 3 \mu(A \cap C) = \mu_C(A).
\end{align*}
The claim for the other marginals follows similarly.

Thus, due to the optimality of $F_\sharp \gamma$, from Theorem \ref{thm:mainparts} we conclude that $F_\sharp \gamma = \gamma_0$ with the measure $\gamma_0$ defined in Theorem \ref{thm:mainparts}. Now, the form of $\spt(\gamma_0)$ forces
\begin{equation}\label{eq:force1}
(\mathtt{p}_2)_\sharp(\gamma\restr{R\times \R \times \R}) + (\mathtt{p}_3)_\sharp(\gamma\restr{R\times \R \times \R}) = \frac12 \mu\restr{C \cup L},
\end{equation}
where $\mathtt{p}_i$ is the projection to the $i$:th coordinate. Similarly,
\begin{equation}\label{eq:force2}
(\mathtt{p}_2)_\sharp(\gamma\restr{L\times \R \times \R}) + (\mathtt{p}_3)_\sharp(\gamma\restr{L\times \R \times \R}) = \frac12 \mu\restr{C \cup R}.
\end{equation}
Together \eqref{eq:force1} and \eqref{eq:force2} imply
\[
(\mathtt{p}_2)_\sharp(\gamma\restr{C\times \R \times \R}) + (\mathtt{p}_3)_\sharp(\gamma\restr{C\times \R \times \R}) = \frac12 \mu\restr{L \cup R}.
\]
Therefore, again from the form of $\gamma_0$, we conclude that for $\mu$-almost every $x\in C$ the disintegration of 
$$(\mathtt{p}_1,\mathtt{p}_2)_\sharp(\gamma\restr{C\times \R \times \R}) + (\mathtt{p}_1,\mathtt{p}_3)_\sharp(\gamma\restr{C\times \R \times \R})$$
 with respect to the projection on the first coordinate gives positive measure for at least 4 points and thus $\gamma\restr{C}$ is not induced by a map from the first coordinate. Since the argument is symmetric with respect to the marginals, we conclude that $\gamma$ is not induced by a map from any of the coordinates.

Let us then show the uniqueness of the symmetric minimizer in \eqref{eq:mainclaimre}. 
First of all, by Proposition \ref{prop:symplans} we already know that there exists a symmetric minimizer.
We will show that the unique symmetric minimizer is the plan $\gamma_1 \in \Gamma_3^{sym}(\mu)$ defined by
\[ 
\gamma_1 = \frac{1}{6}\sum_\sigma \sigma_\sharp\gamma_0,
\]
where the sum is taken over all permutations of the marginals $\sigma \colon \R^{3} \to \R^{3}$.
Since $x_1+x_2+x_3 = 0$ for $\gamma_1$-almost every $(x_1,x_2,x_3)$, by Proposition \ref{prop:optimalRA} the plan $\gamma_1$ is a minimizer.
Now, let $\gamma \in \Gamma_3^{sym}(\mu)$ be a minimizer of \eqref{eq:mainclaimre}. 
By symmetry, we know that 
\[
\sigma_\sharp (\gamma\restr{\sigma^{-1}(C\times R \times L)}) = \gamma\restr{C\times R \times L}
\]
for all permutations $\sigma$ of the marginals. Thus by Lemma \ref{lma:discrete} and the fact that $F_\sharp \gamma = \gamma_0$, we get
\[
 \gamma_0 = \sum_\sigma \sigma_\sharp (\gamma\restr{\sigma^{-1}(C\times R \times L)}) = 6\gamma\restr{C\times R \times L},
\]
implying $\gamma = \gamma_1$.
\end{proof}

\subsection{The higher-dimensional case}\label{subsec:dge1}

Let us  generalize the previous example to the higher dimensions. 
Using the same notation $C$, $L^k$, $R^k$ for $k\in \{1,2\}$ as in the one dimensional case, we write
\[C_d=C^d,~~L_d^{k}=(L^k)^d\,,\text{ and }R_d^{k}=(R^k)^d\,\]
and
\[
L_d = L_d^{1}\cup L_d^2, \text{ and }R_d = R_d^{1}\cup R_d^2.
\]
Let us denote by $\mathcal{L}$  the Lebesgue measure on $\R^d$. 
The marginal measures are  now
\[
 \mu_L = \frac12 \mathcal{L}\restr{L_d}, \qquad \mu_C = 2^d\mathcal{L}\restr{C_d}, \qquad \mu_R = 2^{d-1}\mathcal{L}\restr{R_d}.
\]
The optimal coupling $\gamma_0$ in $\Pi(\mu_C,\mu_R,\mu_L)$ now reads
\[
\gamma_0 = \sum_{k=1}^2\left((H_k)_\sharp 2^{d-1}\mathcal{L}\restr{{C_d}}\right),
\]
where for $k\in\lbrace 1,2\rbrace$, the maps $H_k\colon\R^d \to \R^{3d}$ are given by 
\[T_k(x)=(x,x+(3^k, \dots, 3^k),-2 x - (3^k, \dots, 3^k))\,.\]
We again have the optimality conditition \eqref{eq:optimality}. This gives us that for each minimizer $\gamma$ of $c$ we must have, for each point $(x_1,x_2, x_3)\in \spt(\gamma)$ the equality
\begin{equation*}\label{eq:componentsum0}
x_1+x_2+x_3=0.
\end{equation*}

The uniqueness of optimal $\gamma_0 \in \Pi(\mu_C,\mu_R,\mu_L)$ follows as in the proof of Theorem \ref{thm:mainparts} using the functional $F\colon\Gamma(\mu_{L},\mu_{R})\to \R$
\[F(\omega) 
=\int_{\R^{2d}}\left(|x_1^1+x_3^1|^2+ \cdots +|x_1^d+x_3^d|^2\right)\d\omega(x_1,x_3)\,,\]
and the component-wise anti-monotone transport $\omega_a$ of $\mu_{L}$ to $\mu_{R}$.

The analysis of the minimizers with marginals $\mu = \frac13(\mu_C + \mu_L + \mu_R)$, Lemma \ref{lma:discrete} and the proof of Theorem \ref{thm:mainre}
go as in the one dimensional case, by component-wise considerations.

\bibliography{refsAcademy}

\begin{thebibliography}{10}

\bibitem{AguCar}
{\sc M.~Agueh and G.~Carlier}, {\em Barycenters in the {W}asserstein space},
  SIAM J. on Mathematical Analysis, 43 (2011), pp.~904--924.

\bibitem{BeiHenPen}
{\sc M.~Beiglb{\"o}ck, P.~Henry-Labord{\`e}re, and F.~Penkner}, {\em
  Model-independent bounds for option prices---a mass transport approach},
  Finance and Stochastics, 17 (2013), pp.~477--501.

\bibitem{BeiCoxHue}
{\sc M.~Beiglboeck, A.~Cox, and M.~Huesmann}, {\em Optimal transport and
  {S}korokhod embedding}, Inventiones mathematicae, 208 (2017).

\bibitem{BeiJul}
{\sc M.~Beiglboeck and N.~Juillet}, {\em On a problem of optimal transport
  under marginal martingale constraints}, to appear Ann. Probab.
  ArXiv:1208.1509,  (2012).

\bibitem{Bre91}
{\sc Y.~Brenier}, {\em Polar factorization and monotone rearrangement of
  vector-valued functions}, CPAM, 44 (1991), pp.~375--417.

\bibitem{CorKarLanLee}
{\sc L.~Cort, D.~Karlsson, G.~Lani, and R.~van Leeuwen}, {\em Time-dependent
  {D}ensity-{F}unctional {T}heory for {S}trongly {I}nteracting {E}lectrons},
  Physical Review A, 95 (2017), p.~042505.

\bibitem{CFK17}
{\sc C.~Cotar, G.~Friesecke, and C.~Kl{\"u}ppelberg}, {\em Smoothing of
  transport plans with fixed marginals and rigorous semiclassical limit of the
  {H}ohenberg-{K}ohn functional}, arXiv:1706.05676,  (2017).

\bibitem{CoDePDMa}
{\sc S.~Di~Marino, L.~De~Pascale, and M.~Colombo}, {\em Multimarginal optimal
  transport maps for $1 $-dimensional repulsive costs}, Canadian Journal of
  Mathematics -- Journal Canadien des Math\'ematiques, 67 (2015), pp.~350--368.

\bibitem{DMaGerNen}
{\sc S.~Di~Marino, A.~Gerolin, and L.~Nenna}, {\em Optimal transport theory for
  repulsive costs}, Topological Optimization and Optimal Transport: In the
  Applied Sciences, 17 (2017).

\bibitem{GaSw}
{\sc W.~Gangbo and A.~Swiech}, {\em Optimal maps for the multidimensional
  {M}onge-{K}antorovich problem}, Communications on pure and applied
  mathematics, 51 (1998), pp.~23--45.

\bibitem{GorSeiVig}
{\sc P.~Gori-Giorgi, M.~Seidl, and G.~Vignale}, {\em Density {F}unctional
  {T}heory for strongly interacting electrons}, Physical review letters, 103
  (2009), p.~166402.

\bibitem{Hei2002}
{\sc H.~Heinich}, {\em Probl{\`e}me de {M}onge pour n probabilit{\'e}s},
  Comptes Rendus Mathematique, 334 (2002), pp.~793--795.

\bibitem{LanDMaGerLeeGor}
{\sc G.~Lani, S.~Di~Marino, A.~Gerolin, R.~van Leeuwen, and P.~Gori-Giorgi},
  {\em The adiabatic strictly-correlated-electrons functional: kernel and exact
  properties}, Physical Chemistry Chemical Physics, 18 (2016),
  pp.~21092--21101.

\bibitem{MoaPass}
{\sc A.~Moameni and B.~Pass}, {\em Solutions to multi-marginal optimal
  transport problems concentrated on several graphs}, ESAIM: Control Optim.
  Calc. Var.,  (2017), pp.~551--567.

\bibitem{Pass14}
{\sc B.~Pass}, {\em Multi-marginal optimal transport and multi-agent matching
  problems: uniqueness and structure of solutions}, Discrete Contin. Dyn.
  Syst., 34:1623-1639,  (2014).

\bibitem{PassSurvey}
\leavevmode\vrule height 2pt depth -1.6pt width 23pt, {\em Multi-marginal
  optimal transport: theory and applications}, ESAIM: Math. Model. Numer. Anal.
  (Special issue on ``Optimal transport in applied mathematics"), 49 (2015),
  pp.~1771--1790.

\bibitem{Pra07}
{\sc A.~Pratelli}, {\em On the sufficiency of c-cyclical monotonicity for
  optimality of transport plans}, Mathematische Zeitschrift,  (2007).

\bibitem{DMaGerNenGorSei}
{\sc M.~Seidl, S.~Di~Marino, A.~Gerolin, K.~Giesbertz, L.~Nenna, and
  P.~Gori-Giorgi}, {\em The {S}trictly-{C}orrelated {E}lectron functional for
  spherically symmetric systems revisited}, Accepted in Physical Review A
  (available in arXiv:1702.05022),  (2016).

\bibitem{SeiGorSav}
{\sc M.~Seidl, P.~Gori-Giorgi, and A.~Savin}, {\em Strictly correlated
  electrons in density-functional theory: A general formulation with
  applications to spherical densities}, Physical Review A, 75 (2007),
  p.~042511.

\end{thebibliography}
\bibliographystyle{siam}

\end{document}